\documentclass[12pt,reqno]{amsart}

\newcommand{\be}{\begin{equation}}
\newcommand{\ee}{\end{equation}}
\newcommand{\ben}{\begin{equation*}}
\newcommand{\een}{\end{equation*}}
\newcommand{\ba}{\begin{aligned}}
\newcommand{\ea}{\end{aligned}}

\newtheorem{thm}{Theorem}
\newtheorem{lem}{Lemma}

\newtheorem{ex}{Example}
\newtheorem{prop}{Proposition}

\begin{document}

\makeatletter
\renewcommand{\ps@plain}{%
     \renewcommand{\@oddhead}{\hfil\textrm{Pruitt's Estimates in Banach Space}\hfil\textrm{\thepage}}%
     \renewcommand{\@evenhead}{\@oddhead}%
     \renewcommand{\@oddfoot}{}
     \renewcommand{\@evenfoot}{\@oddfoot}}
\renewcommand{\section}{\@startsection{section}{1}{0em}{\baselineskip}{.5\baselineskip}{\bfseries}}
\numberwithin{thm}{section} \numberwithin{lem}{section}
\numberwithin{cor}{section}\numberwithin{prop}{section}

\makeatother

\title[Pruitt's Estimates in Banach Space]{Pruitt's Estimates in Banach Space}
\author[Griffin]{Philip S. Griffin}
\date{\today}

\thanks{{\it Address:} Department of Mathematics, Syracuse University, Syracuse NY
13244-1150}

\thanks{{\it E-mail:} psgriffi@syr.edu}



\begin{abstract}
Pruitt's estimates on the expectation and the distribution of the
time taken by a random walk to exit a ball of radius $r$ are
extended to the infinite dimensional setting.   It is shown that
they separate into two  pairs of estimates depending on
whether the space is type 2 or cotype 2.  It is further shown
that these estimates characterize
type 2 and cotype 2 spaces.
\end{abstract}
\maketitle

\pagestyle{plain}


\section{Introduction}

\setcounter{equation}{0}
\renewcommand{\theequation}{1.\arabic{equation}}

Let $X,X_1,X_2,\dots$ be a sequence of non-degenerate, independent
and identically distributed random variables taking values in a
separable Banach space $(B, \|\ \|)$.  Set $S_n=\sum_{j=1}^nX_j$
and let $T_r=\min\{n:\|S_n\|> r\}$ be the first time the random
walk leaves the ball of radius $r$. For $r>0$ define
\begin{align}\label{h}
G(r) = P(\|X\|>r),\ K(r) &= {r}^{-2}E(\|X\|^{2};\|X\|\leq r),\notag \\
M(r) ={r}^{-1}E(X;\|X\|\leq r),\ \ & h(r) = G(r)+K(r)+\|M(r)\|.
\end{align}

In the case that $B= \mathbb{R}^d$ with the usual Euclidean norm,
Pruitt \cite {p} obtained two fundamental estimates on the size of
$T_r$ in terms of $h(r)$.  The first gives bounds on the
distribution of $T_r$; there exist constants $0<c<C<\infty$ such
that for all $r>0$ and $n\ge 1$
\begin{equation}\label{1.2}
P(T_r>n)\le\frac{c}{nh(r)},\quad  P(T_r\le n)\le  Cnh(r).
\end{equation}

The second gives bounds on the expectation of $T_r$; there exist
constants $0<c<C<\infty$ such that for all $r>0$
\begin{equation}
\frac{c}{h(r)}\le ET_r\le \frac{C}{h(r)}.  \label{1.3}
\end{equation}

In both cases the constants are independent of the underlying
distribution and of $r$ and $n$, but they do depend on the
dimension $d$. These estimates underly one approach to recent work
on moments of the overshoot and related topics in renewal theory;
see \cite {dm1}, \cite {dm2}, \cite {dm3} and \cite {gmc3} for
examples of this work.

Apart from the problem's intrinsic interest, the motivation for
this paper came from trying to extend
to infinite dimensions some of the work cited above. To take
advantage of the techniques already
developed, infinite dimensional versions of Pruitt's estimates are
needed. We will show that when $B$ is a Hilbert space this can be achieved, but
(\ref{1.2}) and (\ref{1.3}) do not hold in a general
Banach space.
They split into two pairs of estimates depending on
whether $B$ is type 2 or cotype 2. Recall $B$ is type 2 if there
exists a constant $c\in(0,\infty)$ such that for every $n\ge 1$
and every sequence of independent random variables $X_1, \dots ,
X_n$ with $EX_i=0$ and $E\|X_i\|^2<\infty$, $1\le i\le n$, the
following inequality holds;
\begin{equation}
E\|\sum_{i=1}^n X_i\|^2 \le c \sum_{i=1}^n E\|X_i\|^2. \label{1.4}
\end{equation}
We refer to any $c$ for which \eqref{1.4} holds as a type 2 constant for $B$.
$B$ is  cotype 2 if (\ref{1.4}) holds with the inequality
reversed. For example, the $\ell^p$ spaces are  type 2 for $2\le p
< \infty$ and  cotype 2 for $1\le p \le 2$. The main results of
this paper are summarized in the following two Theorems which may be
viewed as giving alternative characterizations of type 2 and cotype 2 spaces;

\begin{thm}\label{T1}
The following are equivalent:
\begin{equation}
B\text{  is type }2;  \label{1t1}
\end{equation}
There is a constant $c$ such that for all distributions X and all
$r>0$
\begin{equation} ET_r\ge \frac{c}{h(r)}; \label{1t2}
\end{equation}
There is a constant $c$ such that for all distributions X, all
$r>0$ and all $n\ge 1$
\begin{equation}
P(T_r\le n)\le  cnh(r). \label{1t3}
\end{equation}
The constants appearing in the three equivalent statements may be taken to depend only on each other.
\end{thm}

\begin{thm}\label{cT1}
The following are equivalent:
\begin{equation}
B\text{  is cotype }2;
\end{equation}
There is a constant $c$ such that for all distributions $X$ and
all $r>0$
\begin{equation} ET_r\le \frac{c}{h(r)};
\end{equation}
There is a constant $c$ such that for all distributions $X$, all
$r>0$ and all $n\ge 1$
\begin{equation}
P(T_r > n)\le  \frac{c}{nh(r)}.
\end{equation}
The constants appearing in the three equivalent statements may be taken to depend only on each other.

\end{thm}

As a consequence it follows that Pruitt's estimates hold precisely when
$B$ is both type 2 and cotype 2, that is by Kwapien's theorem \cite{kw}, when $B$ is isomorphic to a
Hilbert space.  We observe in passing that since $\Bbb R^d$ with the usual norm can be isometrically
embedded in the Hilbert space
$\ell^2$, this implies the constants
in (\ref{1.2}) and (\ref{1.3}) may in fact be taken independent of
dimension.

The method of proof of these results, while similar in places to
Pruitt's original proofs, necessarily differs in important ways
and provides some new insights to the estimates.  For example it
is shown that it suffices to establish the estimates
asymptotically and just for  distributions with mean zero and
finite second moment; see the statements of Theorems \ref{T1a} and
\ref{cT1a}.

It would be interesting to know if there is a different function
which plays the role of $h$ in a general Banach space.  That is,
is there a function $g$ for which (\ref{1.2}) and (\ref{1.3}) hold
with $g$ replacing $h$?  In this direction it is interesting to
note that by Lemma \ref{L1}
below,  if it is the case that $ET_{2r}$ and $ET_{r}$ are comparable,
then such a function exists, namely $g(r) = (ET_r)^{-1}$. The problem
then of course becomes how to calculate $g$ from the underlying
distribution in a manner similar to $h$. This is discussed further in Section 4.


\section{Preliminaries and
General Results}

\setcounter{equation}{0}
\renewcommand{\theequation}{2.\arabic{equation}}

The interplay between the functions $G(r), K(r),$ and $M(r)$ plays
an important role in our analysis. Each of these functions is
right continuous with left limits and approaches $0$ as $r\to
\infty$.  Their behavior near $r=0$ will not be of much
importance, but we point out that $G(r)\to P(\|X\|>0)>0$ as $r\to
0$, hence $h(r)$ is bounded away from $0$ as $r\to 0$.  Further
$h$ is strictly positive for all $r>0$ since an integration by
parts shows that
\begin{equation}
G(r)+K(r)=r^{-2}\int_0^r 2uG(u)\ du. \label{Q2}
\end{equation}
As $\|M(r)\|\le 1$, this also shows that $h(r)\le 2$   for all
$r$.  In dealing with $h$, it is useful to think of $h$ as
decreasing and $r^2h(r)$ as increasing. While this is not quite
true the following inequalities provide suitable substitutes (see
(2.3) of \cite{p}); for any $0<r\le s<\infty$
\begin{equation}
\frac {r^2}{2s^2}\le \frac{h(s)}{h(r)}\le 2. \label{hd}
\end{equation}

We denote by $L$ the collection of random variables taking values
in $B$. The following simple result shows why the estimates in
(\ref{1.2}) take the form they do, and how they relate to
(\ref{1.3}).

\begin{lem}\label{L1}
For any $X\in L$, any $r>0$ and any $n\ge 1$
\begin{equation}
P(T_{2r}\le n)\le \frac{n}{ET_r}, \label{L1a}
\end{equation}
\begin{equation}
P(T_r > n)\le \frac{ET_r}{n}. \label{L1b}
\end{equation}
\end{lem}

\begin{proof} The second inequality is an immediate consequence of Markov's inequality.
For the first we just observe that by the Markov property, for any
$k\ge 0$
$$ P(T_r>kn) \le P(T_{2r}>n)^k.$$
Hence
$$\frac{ET_r}n \le \sum_{k=0}^\infty P(T_r>kn) \le \frac 1{1-P(T_{2r}>n)}$$
from which (\ref{L1a}) immediately follows.
\end{proof}

As mentioned in the introduction, this result raises the
possibility of finding a function which plays the role of $h$ in
the general setting.  It also allows us to show how
(\ref{1.2}) and (\ref{1.3}) split into two separate pairs of estimates.

\begin{lem}\label{L2}
The following are equivalent: There exists a
constant $c$ such that for any $X\in L$ and all $r>0$
\begin{equation}
ET_r\ge \frac{c}{h(r)}; \label{L2b}
\end{equation}
There exists a constant $C$ such that for  all $X\in L$, $r>0$ and $n\ge 1$
\begin{equation}
P(T_r\le n)\le {C}{nh(r)}. \label{L2a}
\end{equation}
The constants appearing in these statements may be taken to depend only on each other.
\end{lem}

\begin{proof} First assume (\ref{L2b}), then by (\ref{L1a}), for all $X\in L$, $r>0$ and $n\ge 1$
\begin{align*}
P(T_r\le n) &\le \frac n{ET_{r/2}} \\
&\le \frac{nh(r/2)}c \\
&\le \frac {8nh(r)}c
\end{align*}
by (\ref{hd}).  Thus we may take $C=8c^{-1}$.

Now assume (\ref{L2a}), then for any $X\in L$ and all $r>0$
\begin{align*}
ET_r &\ge E(T_r;T_r>\frac 1{2Ch(r)}) \\
&\ge \frac 1{2Ch(r)}P(T_r>\frac 1{2Ch(r)}) \\
&\ge \frac 1{4Ch(r)}.
\end{align*}
Thus we may take $c=(4C)^{-1}$.
\end{proof}

\begin{lem}\label{L3}
The following are equivalent: There exists a
constant $c$ such that for any $X\in L$ and all $r>0$
\begin{equation}
ET_r\le \frac{c}{h(r)}; \label{L3b}
\end{equation} There exists a constant $C$ such that for all $X\in L$, $r>0$ and $n\ge 1$
\begin{equation}
P(T_r> n)\le \frac{C}{nh(r)}. \label{L3a}
\end{equation}
The constants appearing in these statements may be taken to depend only on each other.
\end{lem}

\begin{proof} First assume (\ref{L3b}), then (\ref{L3a}) follows immediately from Markov's inequality with $C=c$.

Now assume (\ref{L3a}), then by the Markov property, for any $X\in L$, $r>0$ and $n\ge 3$
\begin{align*}
P(T_r> n) &\le P(T_r> \lfloor n/2\rfloor)P(T_{2r}> \lfloor n/2\rfloor ) \\
&\le \frac {9C^2}{n^2h(r)h(2r)}\\
&\le \frac {72C^2}{n^2h^2(r)}
\end{align*}
by (\ref{hd}).  Hence recalling that $h(r)\le 2$ for all $r$ we
have
\begin{align*}
ET_r &= \sum_{n=0}^\infty P(T_r> n)\\
&\le \sum_{n\le \frac 2{h(r)}+1}1+\sum_{n> \frac 2{h(r)}+1}\frac {72C^2}{n^2h^2(r)}\\
&\le \frac 2{h(r)}+2 +\frac {36C^2}{h(r)} \\
&\le \frac {6}{h(r)}+\frac {36C^2}{h(r)}= \frac c{h(r)}
\end{align*}
where $c=6(1+6C^2)$.
\end{proof}

Recall that $L$ denotes the class of all random variables taking
values in $B$.  We will also have need to consider the following two  subclasses;
\begin{align*}
L^2 &= \{X\in L: E\|X\|^2<\infty\}, \\
L^2_0 &= \{X\in L^2: EX=0\}.
\end{align*}

The definitions of type 2 and  cotype 2 involve sums of
independent random variables in $L^2_0$.  An apparently weaker condition
frequently encountered in the literature is;
$B$ is Gaussian type 2 if there exists a constant $c$
such that for every finite sequence $\{x_j\}_{j=1}^N$ in $B$
\begin{equation}\label{gt2}
E\|\sum_{j=1}^Ng_jx_j\|^2 \le c  \sum_{j=1}^N\|x_j\|^2
\end{equation}
where $g_1,\dots ,g_N$ are real valued IID random variables with standard normal distribution.
$B$ is Gaussian cotype 2 if \eqref{gt2} holds with the inequality reversed.

It is well known that these notions of type 2 (respectively cotype 2) are equivalent and furthermore
the constants in each formulation may be chosen to depend only on each other; see for example pp 53 and 130 of
\cite{ms} combined with Proposition 9.11 of \cite{lt}.

It will be convenient to introduce one more equivalent formulation dealing with sums of IID random
variables in $L^2_0$.  By Propositions 9.19 and 9.20 of \cite{lt}, $B$ is type 2 if
there exists a constant $c$ such that for every $n\ge 1$ and every
sequence of IID random variables $X, X_1, \dots , X_n$ with $X\in
L^2_0$, the following inequality holds;
\begin{equation}
E\|S_n\|^2 \le c n E\|X\|^2. \label{1.4IID}
\end{equation}
$B$ is cotype 2 if the inequality is reversed. We will need the following generalization;

\begin{prop}\label{Astype}
B is type 2 if (and only if) there is a constant $c$ such that for
all  $X\in L^2_0$ \begin{equation}
\limsup_{n\to\infty}\frac{E\|S_n\|^2}{n} \le c  E\|X\|^2.
\label{ATIID1}
\end{equation}
The type 2 constant may be chosen to depend only on the constant $c$ in \eqref{ATIID1}.

B is cotype 2 if (and only if) there is a constant $c$
such that for all  $X\in L^2_0$ \begin{equation}
\liminf_{n\to\infty}\frac{E\|S_n\|^2}{n} \ge c  E\|X\|^2.
\label{ATIID2}
\end{equation}
The cotype 2 constant may be chosen to depend only on the constant $c$ in \eqref{ATIID2}.
\end{prop}

\begin{proof}
Let $x_j\in B$ for $1\le j\le N$.  On some sufficiently rich probability space consider a family
$A_j$ of disjoint sets with $P(A_j)=N^{-1}$ for $1\le j\le N$ and a
Rademacher variable $\varepsilon$ independent of $A_1,\dots, A_N$.  Let $(\varepsilon_i, I_{A_{i,1}},\dots ,I_{A_{i,N}})$ be independent copies of $(\varepsilon, I_{A_1},\dots ,I_{A_N})$ and set
\begin{equation*}
X_i=\varepsilon_i\sqrt{N}\sum_{j=1}^NI_{A_{i,j}}x_j.
\end{equation*}
Then
\begin{equation}\label{X}
E\|X_1\|^2=\sum_{j=1}^N\|x_j\|^2.
\end{equation}
Observe that
\begin{equation*}
\frac{S_n}{\sqrt n}
=\sum_{j=1}^N Z_{n,j}x_j
\end{equation*}
where
\begin{equation*}
Z_{n,j}=\frac 1{\sqrt{n}}\sum_{i=1}^n \varepsilon_i\sqrt{N}I_{A_{i,j}}.
\end{equation*}
By the central limit theorem in $\Bbb R^N$,
\begin{equation*}
(Z_{n,1},\dots ,Z_{n,N})\overset{d}\to (g_1,\dots ,g_N)
\end{equation*}
where $g_1,\dots ,g_N$ are IID random variables with standard normal  distribution.  Thus
\begin{equation*}
\frac{S_n}{\sqrt n}\overset{d}\to\sum_{j=1}^N g_jx_j.
\end{equation*}
Since $\|X_1\|$ is bounded, it then follows from Theorem 5.1 of \cite{dag} that
\begin{equation}\label{gx}
\frac{E\|S_n\|^2}{n}\to E\|\sum_{j=1}^N g_jx_j\|^2.
\end{equation}
Both results now follow from \eqref{X}, \eqref{gx} and the  discussion preceeding  the proposition.
\end{proof}


\section{Main Results}
\setcounter{equation}{0}
\renewcommand{\theequation}{3.\arabic{equation}}

The following inequality related to (\ref{1.4IID}), and which is
valid in any Banach space, will prove useful below;  for any $X\in
L^2$
\begin{equation}\label{kol}
Var(\|S_n\|) \le 4nE\|X\|^2.
\end{equation}
This was first observed by de Acosta \cite{da}.
The proof is based on elementary  properties of conditional expectations  and
martingales. Alternatively \eqref{kol} may be viewed as a simple consequence of the Efron-Stein
inequality, see \cite{st}.

Before stating and proving the  main results, we note that if
$X\in L_0^2$ then
\begin{align*}
r^2\|M(r)\|&=r\|E(X:\|X\|\le r)\|\\ &=r\|E(X:\|X\|> r)\|\\
&\le rE(\|X\|:\|X\| > r)\\ &\le E(\|X\|^2:\|X\| > r)\to 0
\end{align*} as $r\to\infty$, hence by (\ref{Q2})
\begin{equation}
\lim_{r\to\infty}r^2h(r)=E\|X\|^2\quad \text{as }
r\to\infty.\label{hL_0^2}
\end{equation}

\begin{thm}\label{T1a}
The following are equivalent:
\begin{equation}
B\text{  is type }2;  \label{1t1a}
\end{equation}
There is a constant $c$ such that for all $X\in L$
and all $r>0$
\begin{equation} ET_r\ge \frac{c}{h(r)}; \label{1t2a}
\end{equation}
There is a constant $c$ such that for all  $X\in L$,
all $r>0$ and all $n\ge 1$
\begin{equation}
P(T_r\le n)\le  cnh(r); \label{1t3a}
\end{equation}
There is a constant $c$ such that for all  $X\in
L^2_0$
\begin{equation} \liminf_{r\to\infty}\frac{ET_r}{r^2}\ge \frac{c}{E\|X\|^2}; \label{1t4a}
\end{equation}
There is a constant $c$ such that for all  $X\in
L^2_0$
\begin{equation}
\limsup_{r\to\infty}\sup_{n\ge 1}\frac{r^2P(T_r\le n)}{n}\le
cE\|X\|^2. \label{1t5a}
\end{equation}
The constants appearing in these statements may be taken to depend only on each other.
\end{thm}

\begin{proof}
The dependence of the constants on each other follows from the proof.  To avoid repitition this
will not be pointed out explicitly during the proof, but the
reader is asked to keep this in mind.

It follows from Lemma \ref{L2} that
(\ref{1t2a}) is equivalent to (\ref{1t3a}). Clearly
(\ref{1t2a}) implies (\ref{1t4a}) by \eqref{hL_0^2}, while (\ref{1t4a}) implies
(\ref{1t5a}) by (\ref{L1a}).  Thus it suffices to prove
(\ref{1t1a}) implies (\ref{1t3a}) and (\ref{1t5a}) implies
(\ref{1t1a}).

Assume (\ref{1t1a}). Let $Y_k=X_kI(\|X_k\|\le r)$ and
$$ U_n=\sum_{k=1}^n Y_k. $$
Observe that
\begin{equation}\label{T1a1}
\|EU_n\|=n \|EXI(\|X\|\le r)\|=nr\|M(r)\|
\end{equation}
while by (\ref{1t1a})
\begin{align*}\label{T1a2}
E\|U_n-EU_n\|^2 &\le cnE\|Y_1-EY_1\|^2 \\
&\le cnE(\|Y_1\|+E\|Y_1\|)^2 \\
&= cn(E\|Y_1\|^2+3(E\|Y_1\|)^2) \\
&\le 4cnE\|Y_1\|^2 \\
&= 4cnr^2K(r).
\end{align*}
Now if $\|EU_n\|\le r/2$, then
\begin{align*}
P(T_r\le n)&= P(\max_{1\le k\le n}\|S_k\|>r)\\
&\le P(\max_{1\le k\le n}\|X_k\|>r)+P(\max_{1\le k\le n}\|U_k\|>r)\\
&\le nG(r) + P(\max_{1\le k\le n}\|U_k-EU_k\|>r/2)\\
&\le nG(r) + \frac{4E\|U_n-EU_n\|^2}{r^2}\quad \text { by Doob's inequality}\\
&\le nG(r) + 16cnK(r) \\
&\le (16c+1)nh(r).
\end{align*}
On the other hand if $\|EU_n\|\ge r/2$, then by (\ref{T1a1})
\begin{equation*}
P(T_r\le n)\le 1\le \frac {2\|EU_n\|}r = 2n\|M(r)\| \le 2nh(r).
\end{equation*}
Thus in either case (\ref{1t3a}) holds.

Now assume (\ref{1t5a}) and let $X\in L^2_0$. Then for $n$
sufficiently large (depending on $X$)
\begin{align*}
E\|S_n\| &= \int_0^{\infty} P(\|S_n\|>r)\ dr \\
&\le \int_0^{\infty} P(T_r\le n)\ dr \\
&\le \int_0^{\sqrt{cnE\|X\|^2}} \ dr + \int_{\sqrt{cnE\|X\|^2}}^{\infty} \frac {2cnE\|X\|^2}{r^2}\ dr \\
&\le 3\sqrt{cnE\|X\|^2}.
\end{align*}
Hence by (\ref{kol}) for such $n$,
\begin{equation*}
E\|S_n\|^2= Var(\|S_n\|) + (E\|S_n\|)^2 \le 4nE\|X\|^2+ 9cnE\|X\|^2
\end{equation*}
which verifies (\ref{ATIID1}).
\end{proof}

Fix $r>0$ and set
\begin{align}
\hat{X_k} =& X_kI(\|X_k\|\le 3r)+\frac{3rX_k}{\|X_k\|}I(\|X_k\|>3r) \label{hat}\\
\hat{S_n} &= \sum_{k=1}^n \hat{X_k}, \quad \hat{T_r} = \inf\{n:\|\hat{S_n}\|>r\}. \label{Shat}
\end{align}
Since a jump of any size greater than or equal to $3r$ results in
$S_n$ leaving the ball of radius $r$, it follows that
$\hat{T_r}=T_r.$
It is clear from the definition of $\hat{X}$ that
\begin{equation}\label{QX}
E\|\hat{X}\|^2= 9r^2(G(3r)+K(3r)).
\end{equation}
Observe further that
\begin{align}\label{MXhat}
E\|\hat{X}-E\hat{X}\|^2 &\ge E\big (\|\hat{X}\|-\| E\hat{X}\|\big )^2 \notag \\
&= E\|\hat{X}\|^2-2 E\|\hat{X}\|\| E\hat{X}\| + \| E\hat{X}\|^2 \notag \\
&\ge E\|\hat{X}\|^2-6r\| E\hat{X}\|.
\end{align}
while
\begin{align}\label{Mhatt}
\big|\|E\hat{X}\|-\|EXI(\|X\|\le 3r)\|\big| & \le
\|E\hat{X}-EXI(\|X\|\le 3r)\|\notag\\ &\le 3rG(3r).
\end{align}

\begin{thm}\label{cT1a}
The following are equivalent:
\begin{equation}
B\text{  is cotype }2;  \label{c1t1a}
\end{equation}
There is a constant $c$ such that for all  $X\in L$
and all $r>0$
\begin{equation} ET_r\le \frac{c}{h(r)}; \label{c1t2a}
\end{equation}
There is a constant $c$ such that for all  $X\in L$,
all $r>0$ and all $n\ge 1$
\begin{equation}
P(T_r > n)\le  \frac{c}{nh(r)}; \label{c1t3a}
\end{equation}
There is a constant $c$ such that for all  $X\in
L^2_0$
\begin{equation} \limsup_{r\to\infty}\frac{ET_r}{r^2}\le \frac{c}{E\|X\|^2}; \label{c1t4a}
\end{equation}
There is a constant $c$ such that for all  $X\in
L^2_0$
\begin{equation}
 \limsup_{r\to\infty}\sup_{n\ge 1}\frac{nP(T_r > n)}{r^2}\le \frac{c}{E\|X\|^2}. \label{c1t5a}
\end{equation}
The constants appearing in these statements may be taken to depend only on each other.
\end{thm}

\begin{proof} The dependence of the constants on each other is again a consequence of the proof.

It follows from Lemma \ref{L3} that
(\ref{c1t2a}) is equivalent to (\ref{c1t3a}).  Clearly
(\ref{c1t2a}) implies (\ref{c1t4a}) by \eqref{hL_0^2}, while (\ref{c1t4a}) implies
(\ref{c1t5a}) by (\ref{L1b}). Thus it suffices to prove
(\ref{c1t1a}) implies (\ref{c1t3a}) and (\ref{c1t5a}) implies
(\ref{c1t1a}).

Assume (\ref{c1t1a}). Let
\begin{equation*}
\phi(n)= E\max_{1\le k\le n}\big (\|\hat{S_k}-E\hat{S_k}\|\big
)^2.
\end{equation*}
Then
\begin{align}\label{llbb}
\phi(n) &\ge \max_{1\le k\le n}E\big (\|\hat{S_k}-E\hat{S_k}\|\big )^2 \notag\\
&\ge  cnE\|\hat{X}-E\hat{X}\|^2
\end{align}
by (\ref{c1t1a}). By (1.3) of Klass
\cite{k} (alternatively see Theorem 2.3.1 of \cite{dg}),
there exists an absolute constant $\alpha$ such that
\begin{align}\label{K1}
E\max_{1\le k\le T_r}\big (\|\hat{S_k}-E\hat{S_k}\|\big )^2 &\ge \alpha E\phi(T_r)\notag\\
&\ge \alpha cET_r E\|\hat{X}-E\hat{X}\|^2
\end{align}
by \eqref{llbb}.  Since $\|\hat{S_k}\|\le 4r$ for all $1\le k \le T_r$, it follows
that
\begin{equation*}
E\max_{1\le k\le T_r}\big (\|\hat{S_k}-E\hat{S_k}\|\big )^2  \le
E\max_{1\le k\le T_r}\big (\|\hat{S_k}\|+\|E\hat{S_k}\|\big )^2
\le 64r^2.
\end{equation*}
Hence by (\ref{MXhat} ) and (\ref{K1})
\begin{equation}\label{100}
\alpha c  ET_r \big (E\|\hat{X}\|^2-6r\| E\hat{X}\| \big )\le 64r^2.
\end{equation}
We now consider two cases:

\item{Case 1;} $E\|\hat{X}\|^2>12r\| E\hat{X}\|.$

Then by (\ref{100})
\begin{equation*} \label{101}
\alpha c ET_r E\|\hat{X}\|^2 \le 128r^2.
\end{equation*}
On the other hand by (\ref{QX}) and (\ref{Mhatt})
\begin{align*}
9r^2h(3r) &=  E\|\hat{X}\|^2+3r\|EXI(\|X\|\le 3r)\| \\
&\le E\|\hat{X}\|^2+ 3r\| E\hat{X}\|+ 9r^2G(3r) \\
&\le (9/4) E\|\hat{X}\|^2.
\end{align*}
Hence by (\ref{hd})
\begin{equation*} \label{102}
ET_r \le \frac {576}{\alpha ch(r)}.
\end{equation*}
Thus (\ref{c1t3a}) follows by Markov's inequality.

\item{Case 2;} $E\|\hat{X}\|^2\le 12r\| E\hat{X}\|.$

First observe that by (\ref{Mhatt}) and (\ref{QX})
\begin{equation*}
\|EXI(\|X\|\le 3r)\|\le \|E\hat{X}\|+3rG(3r)  \le \|E\hat{X}\|+ \frac {E\|\hat{X}\|^2}{3r}\ \le 5\|E\hat{X}\|,
\end{equation*}
hence
\begin{align}\label{hm}
3rh(3r) &= \frac {\|E\hat{X}\|^2}{3r}+\|EXI(\|X\|\le 3r)\|
\le 9\|E\hat{X}\|.
\end{align}

\item{Subcase 2a;} $n\|E\hat{X}\|\le 3r.$

Then
\begin{equation*}
P(T_r>n) \le 1
\le \frac {3r}{n\|E\hat{X}\|}
\le \frac {9}{nh(3r)} \le \frac {162}{nh(r)}
\end{equation*}
by (\ref{hd}).

\item{Subcase 2b;} $n\|E\hat{X}\|>3r.$

In this case we have
\begin{equation*}
E\|\hat S_n\|\ge \|E\hat S_n\|=n\|E\hat X\|.
\end{equation*}
In particular $E\|\hat S_n\| > 3r$ and so
\begin{align*}
P(T_r>n) &\le P(\|\hat S_n\| \le r) \\
&\le P\big ( 3(E\|\hat S_n\|-\|\hat S_n\|) > 2E\|\hat S_n\|) \\
&\le P\big (3\big | \|\hat S_n\| - E\|\hat S_n\|\big | > 2n\|E\hat{X}\|\big ) \\
&\le\frac{9 Var(\|\hat S_n\|)}{\big (2n\|E\hat{X}\|\big )^2}  \\
&\le\frac{36nE\|\hat X\|^2}{\big (2n\|E\hat{X}\|\big )^2}\quad \text { by (\ref{kol})} \\
&\le \frac {108r}{n\|E\hat{X}\|}\quad \text { by Case 2} \\
&\le \frac {5832}{nh(r)}
\end{align*}
by (\ref{hm}) and (\ref{hd}). Thus (\ref{c1t3a}) holds.

Now assume (\ref{c1t5a}).  Let
\begin{equation*}
M_n=\max_{1\le k\le n}\|S_k\|.
\end{equation*}
Since $P\{M_n\le r\}=\{T_r> n\}$ we have that for sufficiently large $n$ (depending on $X$)
\begin{equation*}
P(M_n^2 \le \frac {nE\|X\|^2}{4c}) \le 1/2.
\end{equation*}
Hence by Doob's inequality, for such $n$
\begin{align*}
4E\|S_n\|^2 &\ge EM_n^2 \\
&\ge E(M_n^2; M_n^2 > \frac {nE\|X\|^2}{4c})\\
&\ge \frac {nE\|X\|^2}{8c}.
\end{align*}
which verifies (\ref{ATIID2}).
\end{proof}


\section{Examples}
\setcounter{equation}{0}
\renewcommand{\theequation}{4.\arabic{equation}}

As was indicated earlier, in order to extend Pruitt's results to
an arbitrary separable Banach space a function other than $h$ is
needed to measure the size of $ET_r$. In this section we consider the
asymptotic behavior of $ET_r$ when $B=\ell^p$ for $2\le p < \infty$.  In particular we obtain sharp
bounds in (\ref{1t4a}) for symmetric  $X\in L^2$
showing that $1/h(r)$ underestimates the size of $ET_r$ when $p>2$.

Let $X,X_1,X_2,\dots$  be a sequence of IID symmetric random variables
taking values in $\ell^p$ and let $X_{i,j}$ denote the $jth$ coordinate
of $X_i$.
In the following $\approx_{p}$ means the ratio of the two quantities is bounded above and below by constants which depend only on $p$.
\begin{lem}\label{est}
Let $\{X_i\}$ be a sequence of IID symmetric random variables
taking values in $\ell^p$ for $1\le p <\infty$. Then
\begin{equation*}\label{EST}
E\|S_n\|^p\approx_{p} \sum_{j=1}^\infty E\big(\sum_{i=1}^n X_{i,j}^2\big)^{p/2}.
\end{equation*}
\end{lem}
The proof is a standard computation using the Khintchine-Kahane inequalities for Rademacher functions, so we will omit it. The upper bound is recorded in Lemma 5.2 of Pisier and Zinn \cite{pz}.  Next we need Rosenthal's inequality (Theorem 1.5.9 of \cite{dg}) from which it follows that for $p\ge 2$
\begin{equation*}
E\big(\sum_{i=1}^n X_{i,j}^2\big)^{p/2}\approx_{p} (nEX_{1,j}^{2})^{p/2} + nE|X_{1,j}|^{p}.
\end{equation*}
Combining these two results we find that for $2\le p < \infty$
\begin{equation*}\label{ss}
E\|S_n\|^p \approx_{p} n^{p/2}\sum_{j=1}^\infty (EX_{1,j}^{2})^{p/2} + n\sum_{j=1}^\infty E|X_{1,j}|^{p}.
\end{equation*}
Now $\ell^p$ is type 2 when $2\le p < \infty$.  Thus by Proposition 9.24 of \cite{lt}, if $X\in L^2$ then $X$ is pregaussian .  In that case let $G(X)$ denote a Gaussian random variable with the same covariance structure as $X$.  Then
\begin{equation}\label{G(X)}
E\|G(X)\|^p = E|g|^p\sum_{j=1}^\infty (EX_{1,j}^{2})^{p/2},
\end{equation}
where $g$ has a standard normal distribution; see p 261 of \cite{lt}.  Since all $L^q$ norms of a Gaussian random variable are comparable (Corollary 3.2 of \cite{lt}), we can conclude that
\begin{equation}\label{festa}
E\|S_n\|^p \approx_{p} n^{p/2}(E\|G(X)\|^{2})^{p/2} + n E\|X\|^{p}.
\end{equation}
This is only useful if $E\|X\|^{p}<\infty$.  But since the constants in (\ref{festa}) do not depend on the random variable X
we can apply this estimate to $\hat{X}$, as defined in (\ref{hat}), and obtain that for any symmetric random variable $X$
\begin{equation}\label{fest}
E\|\hat{S}_n\|^p \approx_{p} n^{p/2}(E\|G(\hat{X})\|^{2})^{p/2} + n E\|\hat{X}\|^{p}.
\end{equation}
Observe also that since $\hat{X_1} = \xi_r X_1$ where
\ben
\xi_r=I(\|X_1\|\le 3r)+\frac{3r}{\|X_1\|}I(\|X_1\| > 3r),
\een
it follows from \eqref{G(X)} and monotone convergence that
\begin{equation}\label{Glt}
E\|G(\hat{X})\|^p\to E\|G({X})\|^p.
\end{equation}
It is not hard to show that  $G(\hat{X})\overset{d}\to G({X})$ and so in fact
$E\|G(\hat{X})\|^q\to E\|G({X})\|^q$ for all $q\ge 0$; see Theorem 3.8.11 of \cite{b}.
However \eqref{Glt}, together with the comparability of Gaussian norms, will
suffice for our needs below.

\begin{prop}\label{examp}
Let $\{X_i\}$ be a sequence of IID symmetric random variables
taking values in $\ell^p$ for $2\le p < \infty$, for which $E\|X\|^2<\infty$. Then there
is a constant $c$, depending only on $p$, such that for every $r>0$
\begin{equation}\label{t2ub}
ET_r \le \frac{cr^2}{E\|G(\hat{X})\|^2}.
\end{equation}
On the other hand there is a constant $c$, again depending only on $p$, such that
\begin{equation}\label{t2lb}
\liminf_{r\to\infty}\frac {ET_r}{r^2}\ge \frac{c}{E\|G({X})\|^{2}}.
\end{equation}
\end{prop}

\begin{proof}
Recalling the definitions of $\hat{S_n}$ and $\hat{T_r}$ in (\ref{Shat}), if we let
\begin{equation*}
\phi(n)= E\max_{1\le k\le n}\|\hat{S_k}\|^p,
\end{equation*}
then by (1.3) of \cite{k} (alternatively see Theorem 2.3.1 of
\cite{dg}) there exists a constant $\alpha_p$, depending only on $p$, such
that
\begin{equation*}\label{eg2b}
E\max_{1\le k\le \hat{T}_r}\|\hat{S_k}\|^{p} \ge
\alpha_pE{\phi}(\hat{T}_r).
\end{equation*}
Now by (\ref{fest}), there exists a constant $c_p$, depending only on $p$, such
that
\begin{equation*}
\phi(n) \ge c_p n^{p/2}(E\|G(\hat{X})\|^{2})^{p/2},
\end{equation*}
hence
\begin{equation*}
(4r)^p \ge \alpha_pc_p ET_r^{p/2} (E\|G(\hat{X})\|^2)^{p/2}.
\end{equation*}
Thus (\ref{t2ub}) follows by Jensen's inequality.

For the lower bound we let
\begin{equation*}
\phi(n)= E\max_{1\le k\le n}\|\hat{S_k}\|^2.
\end{equation*}
By (1.10) of \cite{k1} (alternatively see Theorem
2.3.1 of \cite{dg}) there exists a universal constant $\beta$ such
that
\begin{equation*}\label{eg2}
E\max_{1\le k\le \hat{T}_r}\|\hat{S_k}\|^2 \le \beta E\phi(\hat{T}_r).
\end{equation*}
By Doob's inequality, Jensen's inequality and (\ref{fest}), for some constant $C_p$ depending only on $p$,
\begin{equation*}
\phi(n)\le 4E\|\hat{S_n}\|^2 \le 4(E\|\hat{S_n}\|^p)^{2/p} \le C_p nE\|G(\hat{X})\|^{2} + C_p n^{2/p} (E\|\hat{X}\|^{p})^{2/p}.
\end{equation*}
Thus
\begin{equation}\label{last}
r^2 \le \beta C_p ET_rE\|G(\hat{X})\|^{2} + \beta C_p ET_r^{2/p} (E\|\hat{X}\|^{p})^{2/p}.
\end{equation}
If $p=2$ then $E\|G(\hat{X})\|^{2}=E\|\hat{X}\|^{2}$ by \eqref{G(X)}, and so the following stronger
(than (\ref{t2lb})) condition holds; for all $r>0$
\begin{equation}\label{t2lb1}
\frac {ET_r}{r^2}\ge \frac{c}{E\|G(\hat{X})\|^{2}}
\end{equation}
where $c=(2\beta C_p)^{-1}$.
If $p>2$ then $r^{2-p}E\|\hat{X}\|^{p}\to 0$ since $E\|X\|^2<\infty$.  In that case (\ref{t2lb}) follows  from
(\ref{last}), Jensen's inequality,
and \eqref{Glt} together with the comparability of Gaussian norms.
\end{proof}

When $p=2$ the pair of inequalities \eqref{t2ub} and \eqref{t2lb1} also follow easily from
\eqref{1t2a} and \eqref{c1t2a}.  The main interest in Proposition \ref{examp} as a source of examples is when $p>2$.

Using the bounds in  (\ref{t2ub}) and (\ref{t2lb}) we can compare
$h(r)$ and $ET_r$ for large $r$.  For convenience we will write
$h(r)ET_r\approx L$ if
\begin{equation}\label{T}
cL\le \liminf_{r\to\infty}h(r)ET_r\le
\limsup_{r\to\infty}h(r)ET_r\le CL,
\end{equation}
where the constants depend only on the Banach space.
Thus in the setting of Proposition \ref{examp}, it follows that
\begin{equation}\label{hTc1}
h(r)ET_r\approx  \frac{E\|X\|^2}{E\|G(X)\|^2}.
\end{equation}
The RHS of (\ref{hTc1}) is always bounded below by a constant depending on
the type 2 constant of $B$ by Proposition 9.24 of \cite{lt}. However there is
no corresponding upper bound (unless $p=2$).

We conclude with some specific examples.  Let
\begin{equation*}
X_{i,j}=\alpha_j r_{i,j}I(j\in\Theta_i)
\end{equation*}
where
$\{\Theta_i\}_{i\ge 1}$ is a sequence of IID random variables
taking values in the power set of the natural numbers $\mathbb{N}$,  $\{r_{i,j}\}_{i\ge 1, j\ge
1}$ is a sequence of IID  Rademacher random variables independent of
$\{\Theta_i\}_{i\ge 1}$ and $\{\alpha_j\}_{j\ge 1}$ is a
sequence of non-negative real numbers.  In this case (\ref{hTc1}) becomes
\begin{equation}\label{hT}
h(r)ET_r\approx
\frac{E\big(\sum_{j=1}^\infty \alpha_j^p
I(j\in\Theta)\big)^{2/p}}{\big(\sum_{j=1}^\infty
\alpha_j^pp_j^{p/2}\big)^{2/p}}
\end{equation}
where $p_j=P(j\in \Theta_1)$.

\begin{ex}\label{ex1} $\Theta$ is non-random;  thus $\Theta \equiv A$ for some
fixed subset $A$ of $\mathbb{N}$. \end{ex} In that case $X$ is
simply a Rademacher sum and in order that it take values in $\ell^p$ we
require $\|X\|\equiv\sum_{j\in A}\alpha_j^p<\infty.$  Thus all moments of $\|X\|$ are finite
and in particular
\begin{equation*}
E\big(\sum_{j=1}^\infty \alpha_j^p I(j\in\Theta)\big)^{2/p} =
\big(\sum_{j\in A}\alpha_j^p \big)^{2/p}.
\end{equation*}
Since $p_j\equiv 1$ on $A$ and 0 otherwise,
\begin{equation*} h(r)ET_r\approx 1.
\end{equation*}

\begin{ex} $|\Theta|=k$ for some $k\ge 1$ and $\alpha_i\equiv 1$.
\end{ex}
In that case
\begin{equation*}
\|X\|^p=\sum_{j=1}^\infty I(j\in\Theta)
=k
\end{equation*}
and so again all moments are finite, and
\begin{equation*} h(r)ET_r\approx \frac
{k^{2/p}}{\big(\sum_{j=1}^\infty p_j^{p/2}\big)^{2/p}}.
\end{equation*}
As a particular case assume $1\le k\le d <\infty$ and $\Theta$ is
uniformly distributed on subsets of $\{1, \dots ,d\}$ of size $k$.
Then $p_j= k/d$ for $1\le j\le d $ and is 0 otherwise, hence
\begin{equation*} h(r)ET_r\approx (d/k)^{1-\frac 2p }.
\end{equation*}

This class of examples should be contrasted with Example \ref{ex1} when $A$
is taken to be a subset of $\{1, \dots ,d\}$ of size $k$. If the
$k$ directions in which $S_n$ can move are fixed, then $h$ can be
used to approximate $ET_r$, but if at each step the $k$ directions
are chosen at random from $\{1, \dots ,d\}$ then this is no longer
the case (if $p>2$).  These examples also illustrate that the RHS of (\ref{hTc1}) can not be bounded above
by a constant independent of the distribution of $X$ when $p>2$.

\vskip .2in

\noindent{\bf Acknowledgement.} The author is grateful to Terry
McConnell for several stimulating conversations, to Jim Kuelbs for
providing key references and to Michel Ledoux for insights regarding
Proposition \ref{Astype}.  Thanks also to an anonymous referee
for several useful suggestions.

\end{document}